\documentclass[10pt]{amsart}

\usepackage[margin=3cm]{geometry}

\usepackage[english]{babel}
\usepackage[T1]{fontenc}
\usepackage[utf8]{inputenc}

\usepackage{graphicx}
\usepackage{enumerate}
\usepackage{subfig}
\usepackage{array}
\usepackage{color}
\usepackage{stmaryrd}
\usepackage{amsmath}
\usepackage{amssymb}
\usepackage{latexsym}
\usepackage{mathrsfs}
\usepackage{amsthm}
\usepackage{mathtools}
\usepackage{hyperref}

\newtheorem{theorem}{Theorem}
\newtheorem{corollary}[theorem]{Corollary}

\newtheorem{lemma}[theorem]{Lemma}

\theoremstyle{definition}

\theoremstyle{remark}
\newtheorem{remark}[theorem]{Remark}

\newcommand{\bS}{\mathbb S}

\newcommand{\eps}{\varepsilon}

\DeclareMathOperator{\diver}{div}
\DeclareMathOperator{\grad}{grad}

\usepackage{accents}

\newcommand{\lp}{\left(}
\newcommand{\rp}{\right)}
\newcommand{\lab}{\left|}
\newcommand{\rab}{\right|}
\newcommand{\lb}{\left[}
\newcommand{\rb}{\right]}
\newcommand{\lc}{\left\{}
\newcommand{\rc}{\right\}}

\newcommand{\Lap}{\Delta}

\newcommand{\E}{\mathbb{E}}
\newcommand{\Prob}{\mathbb{P}}

\newcommand{\bR}{\mathbb{R}}
\newcommand{\Zed}{\mathbb{Z}}

\usepackage[dvipsnames]{xcolor}

\DeclareMathOperator{\Cut}{Cut}

\begin{document}

\title[Heat kernel derivative bounds]{Localized bounds on log-derivatives of the heat kernel on incomplete Riemannian manifolds}
\author{Robert W.\ Neel \and Ludovic Sacchelli}
\address{Department of Mathematics, Lehigh University, Bethlehem, Pennsylvania, USA}
\email{robert.neel@lehigh.edu}
\address{Universit\'e C\^ote d'Azur, Inria, CNRS, LJAD, France}
\email{ludovic.sacchelli@inria.fr}

\begin{abstract} 
Bounds on the logarithmic derivatives of the heat kernel on a compact Riemannian manifolds have been long known, and were recently extended, for the log-gradient and log-Hessian, to general complete Riemannian manifolds. Here, we further extend these bounds to incomplete Riemannan manifolds under the least restrictive condition on the distance to infinity available, for derivatives of all orders. Moreover, we consider not only the usual heat kernel associated to the Laplace-Beltrami operator, but we also allow the addition of a conservative vector field. We show that these bounds are sharp in general, even for compact manifolds, and we discuss the difficulties that arise when the operator incorporates non-conservative vector fields or when the Riemannian structure is weakened to a sub-Riemannian structure. 
\end{abstract}

\maketitle

\section{Introduction}

Uniform bounds for the logarithmic derivatives of the heat kernel on a Riemannian manifold, in small time, have been studied since the 90s. In particular, improving on earlier work of Sheu \cite{Sheu}, Hsu \cite{HsuLogDer} and Stroock-Turetsky \cite{StroockTuretsky} independently proved that, on a compact (smooth) Riemannian manifold $M$, there exist two families of positive constants $C_N$ and $D_N$ such that
\begin{equation}\label{Eqn:LogDer}
\lab \nabla^N_x \log p_t(x,y)\rab \leq C_N\lb \frac{d(x,y)}{t} +\frac{1}{\sqrt{t}}\rb^N ,
\end{equation}
and also
\[
\lab \nabla_x^N p_t(x,y)\rab \leq D_N\lb \frac{d(x,y)}{t} +\frac{1}{\sqrt{t}}\rb^N p_t(x,y) ,
\]
both holding for $(t,x,y)\in (0,1]\times M\times M$. (These two families of bounds are widely-known to be equivalent, though, for completeness, we provide the details below.) Here $p_t(x,y)$ is the usual heat kernel (corresponding to the Laplace-Beltrami operator), $\nabla^N$ represents $N$th-order covariant differentiation with $|\cdot|$ the (pointwise) norm on $N$-tensors induced by the Riemannian mertric, and the subscript $x$ indicates that the covariant derivative acts on the first spatial variable (that is, the $x$-variable). In this case, $p_t$ is symmetric, so the result also holds for $y$-derivatives, but the difference in technique for studying the two spatial variables will play a prominent role in what follows, so we emphasize where our derivatives fall throughout.

The work of Sheu, Hsu, and Stroock-Turetsky naturally applies to $x$-derivatives because it considers perturbations of the entire (random) path induced by perturbing the starting point of the diffusion. However, a perturbation of the entire path is in principle a global object, making localization, and thus the extension of these results to non-compact manifolds, non-trivial (at least those without bounded geometry). Indeed, see \cite{Elton-Bismut} for later work in this spirit, looking at more general maps between vector bundles, but still over a compact Riemannian manifold. Recently, Chen-Li-Wu \cite{XM-GenComplete} extended \eqref{Eqn:LogDer} for $N=1,2$ to general complete Riemannian manifolds (that is, complete manifolds with no additional assumptions on the geometry).

Interest in such estimates comes from a few sources. For example, control of the logarithmic gradient (with respect to $x$) of $p_t(x,y)$ is central to the question of whether the Brownian bridge process is a semi-martingale up to its terminal time (which is essentially a question of the integrability of the drift term, given in terms of the logarithmic derivative). While it has long been known that the Brownian bridge is a semi-martingale on a compact Riemannian manifold (see \cite{EltonBook}), extending this result to general complete manifolds was only accomplished a few years ago by G\"{u}neysu \cite{BGun}, which required sufficient localization of the bound on the logarithmic gradient of the heal kernel. More generally, control of the logarithmic gradient and logarithmic Hessian of the heat kernel, on compact manifolds (or complete manifolds with bounded geometry), has been used in various aspects of stochastic analysis on loop space; we refer to the introduction of \cite{XM-GenComplete} for an overview of the literature.

The primary goal of the present paper is to further extend \eqref{Eqn:LogDer} to possibly-incomplete Riemannian manifolds under the weakest condition on the distance to infinity available, for all values of $N$. (This obviously includes all complete manifolds, with no additional conditions.) Moreover, we consider not only the usual heat kernel associated to the Laplace-Beltrami operator, but we allow the addition of a (smooth) conservative vector field. That is, if we let $\Delta_{LB}$ be the Laplace-Beltrami operator, we consider the diffusion operator $\Delta=\Delta_{LB}+Z$ where $Z$ is the gradient vector field of some smooth function on $M$. The point is that $Z$ being conservative is exactly the condition that there exists a (smooth) symmetrizing reference measure, namely a volume measure $\nu$ on $M$ such that the heat kernel is symmetric when written with respect to $\nu$. (See Section \ref{Sect:Background} for more details on the heat kernel with respect to an arbitrary volume.) More precisely, our main result is the following.

\begin{theorem}\label{THM:Main}
Let $M$ be a (possible incomplete) Riemannian manifold, let $\Delta=\Delta_{LB}+Z$, where $\Delta_{LB}$ is the Lapace-Beltrami operator and $Z$ is a smooth vector field, and let $p_t(x,y)$ be the heat kernel for $\Delta$ (with respect to the Riemannian volume measure). Suppose that $Z$ is conservative (of course, this includes the case $Z\equiv 0$), and let $K\subset M\times M$ be compact such that there exists $\eps>0$ so that, for all $(x,y)\in K$, $\{z:d(x,z)+d(z,y)< d(x,y)+\eps\}$ has compact closure. Then there exist sequences of constants $C_N$ and $D_N$ such that, for any $N\geq 1$,
\[\begin{split}
&\lab \nabla_x^N \log p_t(x,y)\rab \leq C_N\lb \frac{d(x,y)}{t} +\frac{1}{\sqrt{t}}\rb^N \\
\text{and }&\lab \nabla_x^N p_t(x,y)\rab \leq D_N\lb \frac{d(x,y)}{t} +\frac{1}{\sqrt{t}}\rb^N p_t(x,y)
\end{split}\]
for all $t\in(0,1]$ and $(x,y)\in K$.
\end{theorem}

The underlying logic of our proof is that $y$-derivatives are fundamentally easier to localize, because the Markov property applies forward in time, meaning perturbing the endpoint of a (random) path only requires perturbing the path in a neighborhood of the endpoint, as opposed to perturbing the entire path. The existence of a symmetrizing measure then allows us to transfer bounds on $y$-derivatives to $x$-derivatives, which, as indicated above, are generally of more interest.

We complement this bound with a natural improvement away from the cut locus. Let $\Cut(M) \subset M\times M$ be the set of $(x,y)$ such that $y$ is in the cut locus of $x$, or equivalently, $x$ is in the cut locus of $y$ (the relationship is symmetric). 

\begin{theorem}\label{THM:NonCut}
Let everything be as in Theorem \ref{THM:Main}, with the additional condition that $K\cap \Cut(M)=\emptyset$. Then, for any $N\geq 1$, 
\[
-4t  \nabla_x^N \log p_t(x,y) -\nabla_x^N d^2(x,y) = O(t) \quad \text{as $t\searrow 0$,}
\]
uniformly over $(x,y)\in K$.
\end{theorem}

A version of this result (with the $O(t)$ replaced by $o(1)$) was proven on a compact Riemannian manifold in the case when $Z\equiv 0$ in \cite{Stroock-Malliavin} (see also \cite{GongMa} for a variant when $N=2$) and then extended to general complete manifolds in \cite{XM-GenComplete}. We are able to localize it and thus extend it to incomplete $M$ just as above.

Of course, these two results can be combined in various ways. To give one example, we have the following.

\begin{corollary}\label{Cor:Main}
Let everything be as in Theorem \ref{THM:Main}. Then there is a positive constant $c$ and a sequence of constants $C_N$ (not necessarily the same as in Theorem \ref{THM:Main}) such that,
\[\begin{split}
| \nabla \log p_t(x,y) | &\leq \begin{cases}
C_1 \lp \frac{d(x,y)}{t}+1\rp & \text{for $d(x,y)\leq c$} \\
\frac{C_1}{t} & \text{for $d(x,y)> c$}
\end{cases} , \\
\text{for any $N\geq 2$,}\quad | \nabla^N \log p_t(x,y) | &\leq \begin{cases}
\frac{C_N}{t} & \text{for $d(x,y)\leq c$} \\
\frac{C_N}{t^N} & \text{for $d(x,y)> c$}
\end{cases}
\end{split}\]
for all $(t,x,y)\in (0,1]\times K$.
\end{corollary}

\begin{remark}\label{Rmk:YDers}
We have stated Theorems \ref{THM:Main} and \ref{THM:NonCut} and Corollary \ref{Cor:Main} for $x$-derivatives, as emphasized, but they hold for $y$-derivatives as well, which is a central feature of the proofs. 
\end{remark}

In Section \ref{Sect:Background}, we describe the geometric context and notation in more detail, including the stochastic aspect. We also recall previous estimates of L\'eandre \cite{LeandreMax} and Bailleul-Norris \cite{IsmaelNorris}, proven in the generality of sub-Riemannian manifolds, which underlie our localization procedure. In Section \ref{Sect:Localizing}, we prove the basic localization results for heat kernel derivatives. In this section, we work in greater generality than just Riemannian manifolds, since the method applies to any diffusion satisfying certain small-time asymptotics, and thus is of independent interest. In particular, the results of this section apply on sub-Riemannian manifolds as well. In Section \ref{Sect:Riemannian}, we apply the general localization results to prove Theorems \ref{THM:Main} and \ref{THM:NonCut}. Finally, in Section \ref{Sect:Misc}, we discuss a few complementary issues. First, we note that the the upper bound given in Theorem \ref{THM:Main} is sharp in the power of $t$ for all $N$, which follows from a recently proven result in \cite{WithLudovic} giving the coefficient of $1/t^N$ in terms of the joint cumulant of geometrically relevant random variables. We illustrate this in the most important case, that of $N=2$, by treating the explicit example of the circle ``by hand.'' Second, we discuss the general (Riemannian) non-compact case, in which no symmetry assumptions are placed on $\Delta+Z$. In this case, while one can still localize $y$-derivatives, it is unclear whether or not the corresponding result for $x$-derivatives holds. This is further related to allowing $\Delta$ to also include a (smooth) potential. Finally, we observe that in the sub-Riemannian situation, even for a symmetric operator on a compact manifold, no estimate of the type \eqref{Eqn:LogDer} is known, because uniform control of the heat kernel, at the level of precision of \eqref{Eqn:LogDer}, near the diagonal is currently unavailable. This is essentially because the diagonal, on a true sub-Riemannian manifold (that is, one which is not Riemannian) is always non-strictly abnormal, and behavior of the heat kernel at abnormals is poorly understood.

As alluded to above, many of the localization ideas we employ also apply to sub-Riemannian manifolds, and this direction is pursued in \cite{WithLudovic}. Moreover, we make use of existing estimates on sub-Riemannian manifolds, though often only because they accommodate the case when $Z\neq 0$ without additional justification. For example, \eqref{Eqn:LogDer} on a compact $M$ for non-zero $Z$ is established in \cite{WithLudovic}. Nonetheless, the eventual results in the Riemannian case are significantly stronger, admit shorter proofs, and draw from different literature, making it best to treat them separately from the sub-Riemannian case. For those wishing to avoid any reference to sub-Riemannian results, we point out that the necessary estimates to prove a weaker version of Theorem \ref{THM:Main}, following the approach in this paper, were essentially already present in the 90s literature on Riemannian manifolds, as discussed in Remark \ref{Rmk:Elton}.


\section{Riemannian background}\label{Sect:Background}

For a (possibly incomplete) Riemannian manifold $M$, we let $\Delta_{LB}$ be the Laplace-Beltrami operator, and we consider the diffusion operator $\Delta=\Delta_{LB}+Z$ where $Z$ is a smooth vector field. We let $X_t$ be the diffusion generated by $\Delta$; in particular, $X_t$ is a Brownian motion (run at $\sqrt{2}$ times the normal speed) with drift. Of course, it is possible that $X_t$ explodes in finite time, and we understand that the process is killed upon explosion.

If we let $\mu$ be the Riemannian volume measure, then $X_t$ has a smooth transition function $p_t(x,y)\in C^{\infty}\lp (0,\infty)\times M\times M\rp$ with respect to $\mu$, that is, for any Borel set $A\subset M$
\[
\Prob\lp X_t\in A | X_0=x\rp = \int_{y\in A} p_t(x,y) \, d\mu(y) .
\]
(Note that the integral over all of $M$ is 1 if and only if the process does not explode, otherwise the integral will be strictly less than 1.) This transition function is also the associated (minimal) heat kernel for $\Delta$ (or fundamental solution to the heat equation), meaning that
\[
\partial_t p_t(x,y) = \Delta_x p_t(x,y) \quad\text{and}\quad \lim_{t\searrow 0} p_t(x,y)= \delta_x(y)
\]
where $\Delta_x$ denotes $\Delta$ acting on the $x$-variable (that is, on the first spatial variable) and $\delta_x$ is the usual point mass at $x$. One consequence is that, for any smooth, compactly supported function $\phi$ on $M$, the corresponding solution to the heat equation with initial values given by $\phi$ is $\phi_t(x) = \int_{y\in M} p_t(x,y)\phi(y)\, d\mu(y)$. (See \cite{McKean} and \cite{Dodziuk} for classical treatments of the existence and basic properties of the heat kernel on incomplete Riemannian manifolds, and \cite{Grigoryan} for a modern survey in the context of weighted Riemannian manifolds.) 

Note that the right-hand side of \eqref{Eqn:LogDer} is bounded from below by a positive constant for $(t,x,y)$ in any compact subset of $(0,\infty)\times M\times M$, so, for any $t_0\in (0,1)$ the estimate holds on any compact subset of $M\times M$ for all $t\in [t_0,1]$ simply by smoothness of the heat kernel. Thus establishing \eqref{Eqn:LogDer} for $t\in(0,1]$ is really a matter of the small-time asymptotics of the heat kernel. Said differently, once we have shown \eqref{Eqn:LogDer} holds uniformly for $t\in (0,t_0)$, we can extend it to $t\in (0,1]$ at the cost of possibly adjusting the constants $C_N$. Our basic approach is to combine a localization result for the heat kernel itself with a global result on $y$-derivatives of the heat kernel in order to localize these derivatives. To provide context for the argument, we first  introduce the necessary notation and describe the known estimates in the Riemannian case.

We will frequently restrict or extend the domain of the heat kernel, and for any open set $U$, we let $p^U_t(x,y)$ be the heat kernel on $U$ (which should be understood with Dirichlet boundary conditions-- that is, it is the transition density for the process killed upon explosion, or equivalently, upon the first exit from $U$). Note that since we are not assuming that our Riemannian manifolds are complete, there is no real distinction between a Riemannian manifold and an open subset of a Riemannian manifold. For a closed set $A$, we let $p_t(x,A,y) = p_t(x,y)-p^{A^c}_t(x,y)$, where $p^{A^c}_t(x,y)$ is extended to be zero for either $x$ or $y$ inside of $A$ (that is, outside of $A^c$). In particular, $p_t(x,A,y)$ gives the contribution to $p_t(x,y)$ from paths that pass through $A$. Further, we define
\[
d(x, A, y) = \inf\{d(x, z) + d(z, y) : z \in A\},
\]
which gives the distance from $x$ to $y$ via paths that pass through $A$. We will also consider the heat kernel with respect to different reference measures, and we will write $p^{\nu}_t(x,y)$ for the heat kernel with respect to a smooth volume measure $\nu$ (meaning a measure with a smooth, positive density in any system of local cooridnates) when there is a chance of confusion. This means that, for any Borel set $A\subset M$
\[
\Prob\lp X_t\in A | X_0=x\rp = \int_{y\in A} p^{\nu}_t(x,y) \, d\nu(y) ,
\]
and thus $d\nu = \frac{d \nu}{d \mu}d\mu$ for a smooth, positive Radon-Nikodym derivative $\frac{d \nu}{d \mu}$.

Assume that $M$ a possibly incomplete Riemannian manifold and $\Delta$ is symmetric, with respect to some smooth reference volume $\nu$. Then Theorem 1.1 of Bailleul-Norris \cite{IsmaelNorris} gives that (here the heat kernel is understood with respect to $\nu$)
\begin{equation}\label{Eqn:BN-A}
\limsup_{t\searrow 0} 4t \log p_t(x,A,y) \leq -d^2(x,A,y)
\end{equation}
uniformly for $x$ and $y$ in any compact subset of $M\setminus \partial A$. Further, under the same assumptions, Theorem 1.2 of \cite{IsmaelNorris} says that
\begin{equation}\label{Eqn:BN-Limit}
\lim_{t\searrow 0} 4t \log p_t(x,y) = -d^2(x,y)
\end{equation}
uniformly for $(x,y)$ in any compact subset of $M\times M$. 

Next, consider a Riemannian metric on $\bR^n$ given by a global orthonormal frame $Z_1,\ldots, Z_n$, where the $Z_i$ are smooth, bounded vector fields with bounded derivatives of all orders. Further, suppose that if we write $\Delta$ as $\sum_{i=1}^n Z_i^2 +Z_0$, then $Z_0$ is also a smooth, bounded vector field with bounded derivatives of all orders. (We don't pursue the question of when a metric admits such a representation, since eventually we localize the following estimate to a small enough neighborhood of a point.) In this case, L\'eandre \cite{LeandreMax} proves that, for any multi-index $\alpha$,
\begin{equation}\label{Eqn:LeandreCoarse2}
\limsup_{t\searrow 0} 4t \log \lp \lab \partial_y^{\alpha}p_t(x,y)  \rab \rp \leq - d^2(x,y) ,
\end{equation}
uniformly for $(x,y)$ in any compact subset of $\bR^n\times \bR^n$. Here $d(x,y)$ is the Riemannian distance, the partial derivatives are understood with respect to the standard Cartesian coordinates, and the result holds for any smooth reference volume on $\bR^n$ (which is clear via the same logic as in Lemma \ref{Lem:Constants} below). Note that this result doesn't require any symmetry of $\Delta$.

We continue with a simple consequence of stochastic calculus for exit times of balls on $M$.

\begin{lemma}\label{Lem:Exit}
Let $A$ be a compact subset of a (possibly incomplete) Riemannian manifold $M$, let $a$ be a positive constant, and let $\sigma_a$ be the first exit time of the diffusion $X_t$ (determined by having infinitesimal generator $\Delta$) from the open ball of radius $a$ around $X_0=x$ (that is, the first time the diffusion moves a distance $a$ from its starting point, which never happens if the diffusion explodes first). Then there exists $T>0$ such that 
\[
\Prob \lp \sigma_a <T | X_0=x \rp <\frac{1}{2}
\]
for any $x\in A$.
\end{lemma}

\begin{proof}
Since $\sigma_a$ is non-decreasing in $a$, without loss of generality we can assume that $a$ is small enough so that an $a$-neighborhood of $A$, briefly $A^a$, is relatively compact, and that $a$ is less than the injectivity radius at all $x\in A$. This follows from the fact that the injectivity radius is continuous and everywhere positive, and thus is bounded from below on a compact. (The continuity of the injectivity radius appears widely known; for example, it is asserted without proof in \cite{Chavel}. For a recent proof, one can see Section 10.8 of \cite{Boumal}.) Thus the squared distance function is smooth on $A\times A^a$.

In particular, for any $x\in A$, let $a$ be fixed as above, and let $r$ be the distance from $x$, and $r_t=r\lp X_t\rp$ be the usual radial process, where $X_0=x$ so that $r_0=0$. Then It\^o's lemma implies there exists some (standard one-dimensional) Brownian motion $W_t$ such that $r_t^2$ satisfies the SDE (written in integral form)
\[
r^2_{t\wedge\sigma_a} = \int_0^{t\wedge\sigma_a} 2r_s \, dW_s +\int_0^{t\wedge\sigma_a} \frac{1}{2}\Lap\lp r^2_s\rp \,ds .
\]

Now by smoothness of the squared distance, there exists some $\alpha>0$ such that the last integral on the right-hand side is bounded from above by $\alpha\lp t\wedge\sigma_a\rp$, and thus also by $\alpha t$, uniformly in $x$. (Indeed, it is well known that $\alpha$ can be given in terms of a lower bound on Ricci curvature and an upper bound on the length of $Z$ over $A^a$.) Moreover, the first integral on the right-hand side is a martingale started from 0, and its quadratic variation is bounded from above by $4a^2 t$, uniformly in $x$. Thus, by Chebyshev's inequality, it converges to 0 in probability as $t\searrow 0$, uniformly in $x$. On the other hand, we have $a^2\Prob^x\lp \sigma_a<t\rp \leq r^2_{t\wedge\sigma_a}$. Putting this together, we have 
\[
\Prob^x\lp \sigma_a<t\rp \leq \frac{1}{a^2}\lp  \int_0^{t\wedge\sigma_a} 2r_s \, dW_s +\alpha t \rp ,
\]
where the right-hand side goes to 0 in probability with $t$ (for any fixed $a>0$), uniformly over $x\in A$. The lemma follows.
\end{proof}


\section{Localizing derivatives of the heat kernel}\label{Sect:Localizing}

For this section, we work in a more general situation than the rest of the paper. In particular, let $M$ be a smooth manifold (of dimension $n$) with a smooth volume $\nu$. We suppose that we have a diffusion $X_t$ on $M$ (that may explode in finite time with positive probability) with a smooth transition density $p_t(x,y)$ with respect to $\nu$; that is, started from any $x\in M$, $X_t$ almost surely has continuous paths, satisfies the strong Markov property (for stopping times with respect to the filtration it generates), and for any Borel $A\subset M$,
\[
\Prob\lp X_t\in A | X_0=x\rp = \int_{y\in A} p_t(x,y) \, d\nu(y) .
\]
We also assume that $d(\cdot,\cdot)$ is a metric on $M$ which is compatible with the smooth manifold topology (so it is also continuous). Finally, we assume that the diffusion satisfies the conclusion of Lemma \ref{Lem:Exit} (where balls are understood with respect to the metric $d$); that is, we assume that the probability of the diffusion from $x$ leaving an open ball of radius $a>0$ around $x$ before time $T$ can be made less than $1/2$ by choosing $T$ small enough, uniformly over all $x$ in a compact. (Recall that the hitting time of a closed set is a stopping time for the filtration generated by $X_t$.) The point is that the basic localization of $y$-derivatives can be established for such a $p_t(x,y)$ using only the type of small-time estimates described in the previous section, and no other aspects of the Riemannian structure.

We begin with some preliminary observations, essentially amounting to smooth calculus, which we collect in the following lemma. For any system of (smooth) coordinates $u_1,\ldots, u_n$ on an open set $U\subset M$ with compact closure, we say that the coordinates are extendable if they can be extended to a neighborhood of the closure of $U$. Note that if we have any two such extendable coordinate systems on $U$, the transition function between them is smooth on a neighborhood of the closure of $U$.

\begin{lemma}\label{Lem:Constants}
Let $M$ and $p_t(x,y)$ be as above. Let $B\subset M$ be an open ball with compact closure that admits (smooth) coordinates $u_1,\ldots, u_n$ that extend to some neighborhood of $\overline{B}$, and let $K\subset M\times M$ be compact, such that $\pi_2(K)$, the projection of $K$ onto the second factor (respectively, $\pi_1(K)$, the projection onto the first factor) is contained in $B$. (Below we understand partial derivatives $\partial^{\alpha}$ on $B$ to be with respect to the $u_i$.)

Then there exists a sequence of positive constants $C_N$ such that, for any $N\geq 1$,
\begin{equation}\label{Eqn:LogDer-Cor}
\lab \partial^{\alpha} \log p_t^{\nu}(x,y)\rab \leq C_N\lb \frac{d(x,y)}{t} +\frac{1}{\sqrt{t}}\rb^N \quad\text{for any multi-index $\alpha$ of weight $N$,}
\end{equation}
uniformly for $(x,y)\in K$, with the derivatives acting in the $y$-variable (respectively, the $x$-variable), if and only if there exists a sequence of positive constants $D_N$ such that, for any $N\geq 1$,
\begin{equation}\label{Eqn:PDer}
\lab \partial^{\alpha} p_t^{\nu}(x,y)\rab \leq D_N\lb \frac{d(x,y)}{t} +\frac{1}{\sqrt{t}}\rb^N p^{\nu}_t(x,y) \quad\text{for any multi-index $\alpha$ of weight $N$,}
\end{equation}
uniformly for $(x,y)\in K$, with the derivatives acting in the $y$-variable (respectively, the $x$-variable).

Further, if either of \eqref{Eqn:LogDer-Cor} and \eqref{Eqn:PDer} hold (with the derivatives acting either on $x$ or on $y$) with respect to some smooth volume $\nu$, then it holds with respect to any smooth volume (that is, for $p_t^{\nu'}$ where $\nu'$ is another smooth volume) with possibly different sequences of constants $C_N$ and $D_N$.
\end{lemma}

\begin{proof}
Assume \eqref{Eqn:LogDer-Cor} holds (for all $N$), with the partials on either $y$ or $x$. For convenience, let $\ell_t(x,y) = \log p_t^{\nu}(x,y)$. Then we have
\[
\partial^{\alpha}  p_t^{\nu}(x,y)= \partial^{n_1}\partial^{n_2}\cdots\partial^{n_N} e^{\ell_t(x,y)}
\]
where the $n_i$ correspond to the indices in $\alpha$, and $N=|\alpha|$. Here the partials are understood to act either all on $y$ or all on $x$. Fa\`a di Bruno's formula gives
\[
\partial^{\alpha}  p_t^{\nu}(x,y)= \sum_{\pi\in\Pi}e^{\ell_t(x,y)} \prod_{B\in\pi}\partial^B \ell_t(x,y)
= p^{\nu}_t(x,y) \sum_{\pi\in\Pi} \prod_{B\in\pi}\partial^B \ell_t(x,y)
\]
where the sum is over all partitions $\pi$ of $\{1,\ldots,N\}$, the product is over all blocks in the partition $\pi$, and $\partial^B$ is understood to mean $\partial^{n_{i_1}}\cdots\partial^{n_{i_{|B|}}}$ where $i_1,\ldots, i_{|B|}$ are the elements of $B$ and $|B|$ is the cardinality of the block $B$. Applying \eqref{Eqn:LogDer-Cor} and using that the sum of the $|B|$ over all $B\in\pi$ is necessarily $N$, we have
\[
\lab \partial^{\alpha}  p_t^{\nu}(x,y)\rab \leq
p^{\nu}_t(x,y) \sum_{\pi\in\Pi} \prod_{B\in\pi} C_{|B|}\lb \frac{d(x,y)}{t} +\frac{1}{\sqrt{t}}\rb^{|B|}
= p^{\nu}_t(x,y) \lb \frac{d(x,y)}{t} +\frac{1}{\sqrt{t}}\rb^N \lp \sum_{\pi\in\Pi} \prod_{B\in\pi} C_{|B|}\rp ,
\]
and thus \eqref{Eqn:PDer} holds with $D'_N= \sum_{\pi\in\Pi} \prod_{B\in\pi} C_{|B|}$.

The other direction is analogous. Assuming that \eqref{Eqn:PDer} holds (for all $N$), and using the same notation as above, we have
\[
\partial^{\alpha} \log p_t^{\nu}(x,y) =
\sum_{\pi\in\Pi} \frac{\lp-1\rp^{|\pi|-1}\lp|\pi|-1\rp!}{\lp p^{\nu}_t(x,y)\rp^{|\pi|}} \prod_{B\in\pi}\partial^B p^{\nu}_t(x,y) ,
\]
so that
\[\begin{split}
\lab\partial^{\alpha} \log p_t^{\nu}(x,y)\rab & \leq   
\sum_{\pi\in\Pi} \frac{\lp-1\rp^{|\pi|-1}\lp|\pi|-1\rp!}{\lp p^{\nu}_t(x,y)\rp^{|\pi|}} \prod_{B\in\pi}p^{\nu}_t(x,y) D_{|B|}\lb \frac{d(x,y)}{t} +\frac{1}{\sqrt{t}}\rb^{|B|} \\
&= \lb \frac{d(x,y)}{t} +\frac{1}{\sqrt{t}}\rb^N \lp \sum_{\pi\in\Pi} \lp-1\rp^{|\pi|-1}\lp|\pi|-1\rp!  \prod_{B\in\pi} D_{|B|}\rp .
\end{split}\]
Thus \eqref{Eqn:LogDer-Cor} holds with $C_N=\sum_{\pi\in\Pi} \lp-1\rp^{|\pi|-1}\lp|\pi|-1\rp!  \prod_{B\in\pi} D_{|B|}$. Again, the computation is valid with $\partial^{\alpha}$ acting on either coordinate.

Next, assume that one, and thus both, of \eqref{Eqn:LogDer-Cor} and \eqref{Eqn:PDer} hold with derivatives on the $y$-variable. Then if $\nu'$ is any other smooth volume, the Radon-Nikodym derivative $\frac{d \nu}{d \nu'}$ is smooth. The product rule gives
\[\begin{split}
\partial_y^{\alpha}  p_t^{\nu'}(x,y) =  \partial_y^{n_1}\partial_y^{n_2}\cdots\partial_y^{n_N} \lp p_t^{\nu}(x,y) \frac{d \nu}{d \nu'}(y) \rp
= \sum_{A\subset {1,\ldots,N}} \lp\partial^A_y p_t^{\nu}(x,y)\rp \lp \partial^{A^c}_y (y)\frac{d \nu}{d \nu'}\rp
\end{split}\]
where the sum is over all subsets of $\{1,\ldots, N\}$ and for each term, the partial derivatives are divided into $A$ and its complement (with $\partial_y^{\emptyset}f(y)=f(y)$). Since we work on a compact $K$, each factor $\partial^{A^c}_y (y)\frac{d \nu}{d \nu'}$ is bounded on $K$. Then using \eqref{Eqn:PDer} on each factor $\partial^A_y p_t^{\nu}(x,y)$, since $|A|\leq |\alpha|=N$, we see that \eqref{Eqn:PDer} holds for $\partial_y^{\alpha}  p_t^{\nu'}(x,y)$ with possibly different constants. By the above, \eqref{Eqn:LogDer-Cor} also holds for $\partial_y^{\alpha}  p_t^{\nu'}(x,y)$, with possibly different constants. Finally, if one, and thus both, of \eqref{Eqn:LogDer-Cor} and \eqref{Eqn:PDer} hold with derivatives on the $x$-variable, then 
\[
\partial_x^{\alpha}  p_t^{\nu'}(x,y) = \lp \partial_x^{\alpha}  p_t^{\nu}(x,y)\rp \frac{d \nu}{d \nu'}(y) ,
\]
and the desired result again follows by smoothness and compactness.
\end{proof}

Next, we show that estimates on the $y$-derivatives of the type in \eqref{Eqn:LeandreCoarse2} can be localized to a neighborhood of the endpoint $y$.

\begin{lemma}\label{Lem:NestedBalls}
For $M$ and $p_t(x,y)$ as described above, suppose that, for some $\eta>0$, $B$, $B'$ and $B''$ are concentric open balls of radii $\eta/2$, $(3/2)\eta$, and $(7/2)\eta$, respectively, and that $B''$ has compact closure. Suppose further that we have an extendable coordinate system on $B''$ and that, for any multi-index $\alpha$, we have
\begin{equation}\label{Eqn:LeandreAgain}
\limsup_{t\searrow 0} 4t \log \lp \lab \partial_y^{\alpha}p_t(x,y)  \rab \rp \leq - d^2(x,y) ,
\end{equation}
uniformly for $x$ and $y$ in $\overline{B''}$ (where the partial derivatives are understood with respect to this coordinate system). Then, for any multi-index $\alpha$,
\begin{equation}\label{Eqn:DerBound}
\limsup_{t\searrow 0} 4t \log \lp \lab \partial_y^{\alpha} p^{B''}_t(x,y)  \rab \rp \leq - d^2(x,y)
\end{equation}
uniformly over $(x,y)\in \overline{B'}\times B$.
\end{lemma}

\begin{proof}

For $(x,y)\in \overline{B'}\times B$ we let $\mu^x_h$ be the (spacetime) hitting measure of $\partial B''$ for the diffusion started from $x$; that is, $\mu^x_h$ is a (sub-) probability measure on $\partial B''\times (0,\infty)$ that describes the first place and time the diffusion from $x$ hits $\partial B''$. Then the strong Markov property implies we have the path decomposition
\begin{equation}\label{Eqn:Decomp}
 p_t^{B''}(x,y) = p_t(x,y) - \int_{(z,\tau)\in \partial B''\times (0,t)} p_{t-\tau}(z,y) \, d\mu^x_h .
\end{equation}
Turning our attention to derivatives, since $p_t$ is smooth in $y$ for all positive times, we have
\[
\partial_y^i \int_{(z,\tau)\in \partial B''\times (0,t)} p_{t-\tau}(z,y) \, d\mu^x_h =
\int_{(z,\tau)\in \partial B''\times (0,t)} \partial_y^i p_{t-\tau}(z,y) \, d\mu^x_h 
\]
for any $i$. Iterating this, we have that, for any multi-index $\alpha$,
\begin{equation}\label{Eqn:DerInside}
\partial_y^{\alpha} \int_{(z,\tau)\in \partial B''\times (0,t)} p_{t-\tau}(z,y) \, d\mu^x_h =
\int_{(z,\tau)\in \partial B''\times (0,t)} \partial_y^{\alpha} p_{t-\tau}(z,y) \, d\mu^x_h ,
\end{equation}
for all $y\in B$ and any $x$.

From \eqref{Eqn:LeandreAgain} and the construction of $B''$, for any $\delta>0$, the integrand in the right-hand side of \eqref{Eqn:DerInside} has absolute value bounded by $e^{-(3\eta)^2+\delta}$ once $t$ is small enough, uniformly in $(z,y)\in \partial B''\times B$. Using this, and  applying $\partial^{\alpha}_y$ to both sides of \eqref{Eqn:Decomp}, we find that, for any multi-index $\alpha$,
\[
\lab \partial_y^{\alpha} p_t^{B''}(x,y) - \partial_y^{\alpha} p_t(x,y) \rab \leq e^{-(3\eta)^2+\delta} 
\]
for all $(x,y)\in \overline{B'}\times B$, if $t$ is small enough. On the other hand, $\partial_y^{\alpha} p_t(x,y)$ term satisfies \eqref{Eqn:LeandreAgain} with $d^2(x,y)\leq (2\eta)^2$ for $(x,y)\in  \overline{B'}\times B$, so its contribution to the bound on $\partial_y^{\alpha} p_t^{B''}(x,y)$ is more significant by an exponential factor (which is the underlying idea of localizing the estimate). At any rate, it follows that 
\[
\limsup_{t\searrow 0} 4t \log \lp \lab \partial_y^{\alpha} p^{B''}_t(x,y)  \rab \rp \leq - d^2(x,y)
\]
uniformly over $(x,y)\in \overline{B'}\times B$, and the result is proven.
\end{proof}

Finally, we show how to combine this local result with a localization result for the heat kernel itself to localize derivatives.

\begin{lemma}\label{Lem:Uc}
In the situation described above, suppose that, $K_0\subset U_0\subset K_1\subset U_1\subset M$ where $K_0$ and $K_1$ are compact and $U_0$ and $U_1$ are open with compact closure, and suppose that, for some $\eps>0$, the heat kernel on $M$ satisfies the estimate
\begin{equation}\label{Eqn:ThroughUc}
\limsup_{t\searrow 0} 4t \log p_t\lp x,U_1^c,y\rp \leq -\lp d(x,y)+\eps\rp ^2
\end{equation}
uniformly for $x$ and $y$ in $K_1$.

Suppose further that there exists $\eta>0$ and $y_0\in K_0$, such that, if we consider the concentric balls $B, B', B''$ centered at $y_0$ (as in Lemma \ref{Lem:NestedBalls}), the closure of $B''$ is contained in $U_0$, and the conclusion of Lemma \ref{Lem:NestedBalls} holds (including the existence of an appropriate extendable coordinate system). Then for any multi-index $\alpha$,

\begin{equation}\label{Eqn:WithEta}
\limsup_{t\searrow 0}  4t \log\lp  \lab\partial_y^{\alpha} p_t(x,U_1^c,y)\rab \rp \leq \lp d(x,y) +\eps-3\eta \rp^2
\end{equation}
uniformly over $x\in K_0$ and $y\in B$.
\end{lemma}

\begin{proof}
Let $y_0\in K_0$ and the corresponding concentric balls $B$, $B'$, and $B''$ be as in the lemma. Let $\sigma$ be the first hitting time (of the diffusion $X_t$) of $\partial U_1$, let $\tau_1$ be the first hitting time of $\partial B'$ after $\sigma$, and let $\tau'_1$ be the first hitting time of $\partial B''$ after $\tau_1$. Then, for $i>1$, we iteratively define $\tau_i$ to be the first hitting time of $\partial B'$ after $\tau'_{i-1}$ and $\tau'_i$ the first hitting time of $\partial B''$ after $\tau_i$. In the notation of \eqref{Eqn:BN-A}, we have
\[
\partial_y^{\alpha} p_t(x,U_1^c,y) = \partial_y^{\alpha} p_t(x,y)- \partial_y^{\alpha} p^{U_1}_t(x,y) .
\]
We understand $p^{B''}_t(x,y)$ to be zero whenever $x\not\in B''$ or $t<0$, and we let $\mu^x_i$ be the (spacetime) distribution of $\lp \tau_i, X_{\tau_i}\rp$ when the diffusion is started from $x$, so that in particular, $\mu^x_i$ is a sub-probability measure on $[0,\infty) \times\partial B'$. Then we have the path decomposition
\begin{equation}\label{Eqn:MuX}
p_t(x,U_1^c,y)= \sum_{i=1}^{\infty} \int_{(z,\tau) \in\partial B'\times(0,t)}p^{B''}_{t-\tau}\lp z,y \rp \, d\mu^{x}_i ,
\end{equation}
which is valid for any $x\in K_0$ and $y\in B$.

We now bound the distributions of the $\tau_i$, starting with $\tau_1$. For any small enough $\delta$, let $V$ be a $\delta$-neighborhood of $\partial B'$. By the conclusion of Lemma \ref{Lem:Exit}, which we assume holds, there is a $t_0$, such that, for any $z\in K_1$, the probability that the diffusion, started from $z$, leaves the ball of radius $\delta$ centered at $z$ before $t_0$ is less than $1/2$. It follows that, for any $t<t_0$, for any $x\in K_0$,
\[
\frac{1}{2} \Prob^x\lp \tau_1 \leq t \rp \leq \int_V p_t\lp x,U^c,z\rp \, d\nu(z) .
\]
On the other hand, by the triangle inequality and the definitions of $U_1$, $\sigma$, and $\tau_1$, we have $d(x,U_1^c,z)\geq d\lp x,\partial B' \rp +\eps-\delta$ for all $z\in V$ and $x\in K_1$. Note that $\overline{B''}$ is a compact subset of $U_0$, and thus for small enough $\delta$, \eqref{Eqn:ThroughUc} holds for $y\in V$. So for all small enough $t$,
\[
\int_V p_t(x,z) \, d\nu(z) \leq \nu(V) \exp\lb -\frac{\lp d(x,\partial B') +\eps-2\delta \rp^2}{4t} \rb ,
\]
uniformly in $x\in K_0$. Since $\nu(V)$ is positive for all $\delta>0$, we conclude that
\begin{equation}\label{Eqn:Tau1}
\limsup_{t\searrow 0} 4t \log \Prob^x\lp \tau_1 \leq t \rp \leq - \lp d(x,\partial B') +\eps \rp^2
\end{equation}
uniformly over $x\in K_0$.

By the strong Markov property, the distribution of $\tau_i-\tau'_{i-1}$ depends only on $X_{\tau'_{i-1}}$; in particular, it doesn't depend on $i$, and similarly for the distribution of $\tau'_i-\tau_{i}$. Moreover,  $\tau_i-\tau'_{i-1}$ is conditionally independent of all of the previous increments $\tau_j-\tau'_{j-1}$ and $\tau'_{j-1}-\tau_{j-1}$ for $j <i$ given $X_{\tau'_{i-1}}$, and similarly for $\tau'_i-\tau_{i}$. Now, in order for $\tau_i$ to be less than $t$, $\tau_i-\tau'_{i-1}$ and $\tau'_{i-1}-\tau_{i-1}$ must both be less than $t$. Further, $\tau_i-\tau'_{i-1}$ can only be less than $t$ if the shifted diffusion $\{X_{\tau'_{i-1}+t}, t\geq 0 \}$ moves a distance $2\eta$ from its starting point (at $X_{\tau'_{i-1}}\in \partial B''$) in time $t$ or less, and similarly, $\tau'_{i-1}-\tau_{i-1}$ can only be less than or equal to $t$ if it moves a distance $2\eta$ in time $t$ or less. By independence and the conclusion of Lemma \ref{Lem:Exit}, it follows that there exists $t_0$ such that, for $t<t_0$ and any $i\geq 2$,
\[
\Prob^x\lp \tau_i-\tau_1<t \rp<\lp\frac{1}{4}\rp^{i-1}
\]
for all $x\in K_0$.

Combining this with \eqref{Eqn:Tau1} and again using the independence, we see that,
for any $\delta>0$ there exists $t_0>0$ such that 
\[
\Prob^x\lp \tau_i \leq t \rp \leq \lp \frac{1}{4}\rp^{i-1} \exp\lb - \frac{\lp d(x,\partial B') +\eps-\delta \rp^2}{4t}\rb
\]
whenever $0<t<t_0$, for all $x\in K_0$ and all $i=1,2,3,\ldots$. Returning to the spacetime hitting measures $\mu^x_i$ in \eqref{Eqn:MuX}, we see that if we consider their restrictions to $[0,t]$ (for $t<t_0$), their masses are absolutely summable, by the above and the summability of the geometric series. More precisely, if we let $\mu^{x,t}$ denote their sum, then it is a finite measure on $\partial B'\times[0,t]$ with total mass bounded from above by
\begin{equation}\label{Eqn:FBound}
F(t)=\frac{1}{2} \exp\lb - \frac{\lp d(x,\partial B') +\eps-\delta \rp^2}{4t}\rb
\end{equation}
for any $x\in K_0$. That is, if $F^{x,t_0}(t)$ is the distribution function of the time-marginal of $\mu^{x,t_0}$, viewed as a measure on $[0,t_0]$ indexed by $x$, then $F(t)$ is an upper bound on $F^{x,t_0}(t)$ (for all $t\in [0,t_0]$) which holds uniformly in $x$. Combining this with \eqref{Eqn:MuX} and the smoothness of $p^{B''}_{t}\lp z,y \rp$ in $y\in B$ for fixed $t>0$ (as elaborated on in the proof of Lemma \ref{Lem:NestedBalls}), the dominated convergence theorem allows us to show that
\[
\partial_y^{\alpha} p_t(x,U_1^c,y)= \int_{(z,\tau) \in\partial B'\times(0,t)} \partial_y^{\alpha} p^{B''}_{t-\tau}\lp z,y \rp \, d\mu^{x,t} .
\]
Using \eqref{Eqn:DerBound}, we see that, for any $\delta$, for all small enough $t$ (uniformly in $x\in K_0$ and $y\in B$),
\[\begin{split}
\lab\partial_y^{\alpha} p_t(x,U_1^c,y)\rab &\leq \int_{(z,\tau) \in\partial B'\times(0,t)} \exp\lb -\frac{(d(\partial B',y) -\delta)^2}{4(t-\tau)}\rb  \, d\mu^{x,t} \\
&= \int_{\tau \in(0,t)} 
\exp\lb -\frac{(d(\partial B',y) -\delta)^2}{4(t-\tau)}\rb dF^{x,t_0}(\tau) ,
\end{split}\]
where we've used the fact that the integrand doesn't depend on $z\in \partial B'$ to reduce to the time-marginal of $\mu^{x,t}$. Because the integrand is monotone decreasing, we get an upper bound from the upper bound \eqref{Eqn:FBound} on the distribution function, so that, for some $C>0$
\[\begin{split}
&\lab\partial_y^{\alpha} p_t(x,U_1^c,y)\rab \leq  \int_{\tau \in\times(0,t)}  \exp\lb -\frac{(d(\partial B',y) -\delta)^2}{4(t-\tau)}\rb dF(\tau) \\
&=  \int_{\tau \in(0,t)} \exp\lb -\frac{(d(\partial B',y) -\delta)^2}{4(t-\tau)}\rb \exp\lb - \frac{\lp d(x,\partial B') +\eps-\delta \rp^2}{4t}\rb  \frac{\lp d(x,\partial B') +\eps-\delta \rp^2}{8t^2}  \, d\tau \\
& \leq C \int_{\tau \in(0,t)} \exp\lb -\frac{(d(\partial B',y) -\delta)^2}{4(t-\tau)}\rb \exp\lb - \frac{\lp d(x,\partial B') +\eps-2\delta \rp^2}{4t}\rb  \, d\tau ,
\end{split}\]
where, in the final line, we've absorbed the $(( d(x,\partial B') +\eps-\delta )^2)/(8t^2)$ into the second exponential at the cost of doubling $\delta$ and multiplying by a constant.

Note that, for any constants $a,b,t>0$ the maximum of 
\[
-\frac{a^2}{4(t-\tau)}-\frac{b^2}{4\tau} \quad\text{for $\tau\in (0,t)$}
\]
is $-(a+b)^2/(4t)$ (achieved at $\tau= t/(1+(a/b))$). Then, applying this to the integrand above and using that $\delta>0$ is arbitrary to clean up the constants, we have that, for any $\delta>0$, there exists $t_0>0$ such that, for all $t\in (0,t_0)$, all $y\in B$, and all $x\in K_0$,
\[
\lab\partial_y^{\alpha} p_t(x,U_1^c,y)\rab \leq  
\exp\lb - \frac{\lp d(x,\partial B')+ d(\partial B',y) +\eps-\delta \rp^2}{4t}\rb .
\]
Since the diameter of $B'$ is at most $3\eta$, we conclude that, for every $\alpha$,
\[
\limsup_{t\searrow 0}  4t \log\lp  \lab\partial_y^{\alpha} p_t(x,U_1^c,y)\rab \rp \leq \lp d(x,y) +\eps-3\eta -\delta \rp^2
\]
uniformly over $(x,y)\in K_0 \times B$. Since $\delta$ is arbitrary, we can send it to zero, completing the proof.
\end{proof}


\section{Proving the Riemannian bounds}\label{Sect:Riemannian}

We now return to the situation of Theorems \ref{THM:Main} and \ref{THM:NonCut}; namely, where $M$ is a (possibly incomplete) Riemannian manifold, $\Delta=\Delta_{LB}+Z$ where $Z$ is conservative, $d(\cdot, \cdot)$ is the Riemannian distance, and $p_t(x,y)$ is the corresponding heat kernel with respect to the Riemannian volume. This is a special case of the situation considered in the previous section, so those results apply here.

As is clear from the previous two sections, we mostly work with partial derivatives in coordinates. But since we eventually want to give an estimate for covariant derivatives, we now relate the two. Since covariant differentiation is tensorial in the vector fields along which we differentiate, it's essentailly enough to look at the coordinate vector fields.

\begin{lemma}\label{Lem:Cov}
Let $B\subset M$ be an open ball in $M$ with compact closure and extendable coordinates $\partial^1,\ldots,\partial^n$, and suppose that $f$ is a smooth function on a neighborhood of $\overline{B}$. Then if $I=(i_1,\ldots,i_N)$ is a sequence of indices from ${1,\ldots,n}$, we have
\[
\nabla^N_{\partial^{I}} f (x) = \nabla^N_{\partial^{i_1},\ldots,\partial^{i_N}} f (x) = \partial^{i_1}\cdots\partial^{i_N} f(x) +\sum_{J:J\subsetneq I} F^{I,J}(x) \partial^J f(x)
\]
where the sum is over subsequences $J$ of $I$ which are neither empty nor all of $I$, and each $F^{I,J}$ is a smooth function on a neighborhood of $\overline{B}$ depending on the Christoffel symbols (of the Levi-Civita connection on $M$). 

Further, there exists $c>0$ (depending only on $N$, $B$, and the choice of coordinates), such that, for any collection $V_1,\ldots,V_N$ of (smooth) vector fields of length 1 or less,
\[
\nabla^N_{V_1,\ldots,V_N} f (x) = \sum_{I\in \{1,\ldots,n\}^N} a_I(x) \nabla^N_{\partial^{I}} f (x) ,
\]
where the $a_I$ are smooth functions on a neighborhood of $\overline{B}$, indexed by all sequences of indices $I$ and depending on the $V_i$, such that $|a_I(x)|<c$ for all $I$ and all $x$ in a neighborhood of $\overline{B}$.
\end{lemma}

\begin{proof}
The first result is trivial for $N=1$ and well known for $N=2$, where the formula for the $(0,2)$-version of the Hessian in coordinates is
\[
\nabla^2_{\partial_j,\partial_i} f = \partial_j\partial_i f -\sum_{k=1}^n\Gamma^k_{ji}\partial_k f .
\]
We now proceed by induction, assuming the result for any $\nabla^N_{\partial^{I}}$. Since $\nabla^N_{\partial^{I}}$ is a $(0,N)$-tensor, we compute its covariant derivative as
\[
\nabla_{\partial_j}\nabla^N_{\partial^{I}} f = \partial_j \nabla^N_{\partial^{I}} f -
 \nabla^N_{\nabla_{\partial_j}\partial^{i_1},\partial^{i_2}\ldots,\partial^{i_N}} f -\cdots- 
  \nabla^N_{\partial^{i_1}\ldots,\partial^{i_{N-1}}, \nabla_{\partial_j}\partial^{i_N}} f 
\]
Then using the definition of the Christoffel symbols and multilinearity, we have
\[
\nabla_{\partial_j}\nabla^N_{\partial^{I}} f = \partial_j \nabla^N_{\partial^{I}} f -
\sum_{k=1}^n\lb \Gamma_{j i_1}^k\nabla^N_{\partial^k,\partial^{i_2}\ldots,\partial^{i_N}} f -\cdots- 
\Gamma^k_{j i_N}  \nabla^N_{\partial^{i_1}\ldots,\partial^{i_{N-1}}, \partial^{k}} f \rb .
\]
Using the induction assumption on all $N$th-order covariant derivatives on the right-hand side and basic calculus, the result follows for  $\nabla_{\partial_j}\nabla^N_{\partial^{I}} f$. But since all $N+1$st-order covariant derivatives under consideration are of this form for some $j$ and $I$, the first claim follows.

For the second, since the metric in the coordinate system is smooth and uniformly positive definite, any unit vector can be written as a linear combination of coordinate vector fields with uniformly bounded coefficients. Then the claim follows by linear algebra.
\end{proof}

\begin{proof}[Proof of Theorem \ref{THM:Main}]
Note that the condition on $K$, namely that $\{z:d(x,z)+d(z,y)< d(x,y)+\eps\}$ has compact closure for all $(x,y)\in K$, is symmetric in $x$ and $y$, that is, it is preserved under reflection across the diagonal in $M\times M$. Therefore, we can enlarge $K$ by taking the union with its reflection to get a symmetric set which satisfies the same condition, and obviously it's enough to prove the theorem for this (potentially) larger set. Thus, without loss of generality, we assume that $K$ is symmetric for the rest of the proof.

Because $Z$ is conservative, there exists a smooth measure $\nu$ such that $p_t^{\nu}(x,y)$ is symmetric. To elaborate, recall that $\Delta_{LB}=\diver_{\mu}\circ\grad $, where $\diver_{\mu}$ is the divergence with respect to the Riemannian volume $\mu$ and $\grad$ is the Riemannian gradient. Any other smooth volume $\nu$ can be written as $e^f\mu$ for some smooth $f$, and then we can write our operator in divergence form with respect to $\nu$ as
\[
\Delta = \diver_{\nu}\circ\grad -\grad f +Z ,
\]
where $\grad f$ is a vector field acting on smooth functions. Since $\diver_{\nu}\circ\grad$ is symmetric with respect to $\nu$ and $Z-\grad f$ is only when it is identically 0, we see that $\Delta$ admits a symmetrizing volume if and only if $Z=\grad f$ for some smooth $f$, in which case the symmetrizing measure is $\nu =e^f\mu$ (unique up to constant rescaling). Thus, we now let $\nu$ be this measure for the given $Z$, and we note that $p_t^{\nu}$ satisfies the assumptions of the previous section and is also symmetric.

Let $y_0$ be any point on $M$. Then there exists some $\eta$ such that the concentric balls $B$, $B'$, and $B''$ around $y_0$ are such that $B''$ has compact closure and is contained in a single extendable coordinate system. Using a bump function, $B''$ can be isometrically included in a Riemannian manifold $M'$ diffeomorphic to $\bR^n$ with a Riemannian metric and smooth volume that satisfy the assumptions under which L\'eandre proved \eqref{Eqn:LeandreCoarse2}. In particular, \eqref{Eqn:LeandreAgain} is satisfied for $p_t^{\nu,M'}(x,y)$ uniformly for $(x,y)$ in any compact subset of $M'\times M'$, and thus for $x$ and $y$ in $\overline{B''}$. It follows, by Lemma \ref{Lem:NestedBalls}, that, for
any multi-index $\alpha$,
\begin{equation}\label{Eqn:DerBoundNu}
\limsup_{t\searrow 0} 4t \log \lp \lab \partial_y^{\alpha} p^{\nu,B''}_t(x,y)  \rab \rp \leq - d^2(x,y)
\end{equation}
uniformly over $(x,y)\in \overline{B'}\times B$. But since $p^{\nu,B''}_t(x,y)$ only depends on (the Riemannian structure and the volume on) $B''$, it this holds for $p^{\nu,B''}_t(x,y)$ with $B''$ viewed as an open ball in the original $M$ as well. Also note that, for fixed $y_0$, if \eqref{Eqn:DerBoundNu} holds for some $\eta$, then it also holds for any smaller $\eta$.

Now take any $(x_0,y_0)\in K$. By the triangle inequality and continuity of the distance, there is some $\delta>0$ such that, if $N_{x_0}(\delta)$ is an open $\delta$-ball around $x_0$ and $N_{y_0}(\delta)$ is an open $\delta$-ball around $y_0$ (here we use $N$ instead of $B$ for these balls to avoid notational confusion with $B$, $B'$, and $B''$), then
\[
U_1= \bigcup_{x\in N_{x_0}(\delta),y\in N_{y_0}(\delta)} \lc z:d(x,z)+d(z,y)< d(x,y)+\frac{\eps}{2}\rc
\]
is open with compact closure (because it is a subset of $\{z:d(x_0,z)+d(z,y_0)< d(x_0,y_0)+\eps\}$ for small enough $\delta$). Now let
\[\begin{split}
K_0 &=   \overline{N_{x_0}(\delta/2)} \cup \overline{N_{y_0}(\delta/2)} , \\
U_0 &=   N_{x_0}(\delta) \cup N_{y_0}(\delta) , \\
\text{and}\quad  K_1 &=  \overline{U_0}.
\end{split}\]
Then $K_0\subset U_0\subset K_1\subset U_1\subset M$ where $K_0$ and $K_1$ are compact and $U_0$ and $U_1$ are open with compact closure. Further, again by the triangle inequality, we have that there exists $\eps'>0$ such that
\begin{equation}\label{Eqn:U1}
d\lp x,U^c_1,y\rp > d(x,y)+\eps'
\end{equation}
for all $x$ and $y$ in $K_1$. Because $\Delta$ is symmetric with respect to $\nu$ and $K_1$ is a compact subset of $U_1$, we can apply the result of Bailleul-Norris, so that \eqref{Eqn:BN-A} yields
\[
\limsup_{t\searrow 0} 4t \log p^{\nu}_t(x,U^c_1,y) \leq -\lp d(x,y)+\eps'\rp^2
\]
uniformly for $x$ and $y$ in $K_1$. After possibly shrinking $\eta$, we have that $\overline{B''}\subset U_0$, which means we've shown that all of the assumptions of Lemma \ref{Lem:Uc} are satisfied. Further, after possibly shrinking $\eta$ again, we have $3\eta<\eps'/2$. We conclude that, for any multi-index $\alpha$,
\begin{equation}\label{Eqn:Local}
\limsup_{t\searrow 0}  4t \log\lp  \lab\partial_y^{\alpha} p^{\nu}_t(x,U_1^c,y)\rab \rp \leq -\lp d(x,y) +\eps'-3\eta \rp^2 < -\lp d(x,y) +\frac{\eps'}{2} \rp^2
\end{equation}
uniformly over $x\in K_0$ and $y\in B$.

Because $U_1$ has compact closure, we can isometrically include it in a compact Riemannian manifold $M'$ (say, by a smooth doubling construction), equipped with a measure that extends $\nu$ and a conservative vector field that extends $Z$; we just extend the smooth $f$ in $\nu=e^f\mu$, and the associated heat kernel also remains symmetric. Moreover, we can choose $M'$ such that \eqref{Eqn:U1} holds with respect to the distance on $M'$ as well. (In particular, we can use a bump function to conformally rescale the metric outside of $U_1$ to make paths from $x$ to $y$ that leave $U_1$ ``too long,'' and for all $(x,y)\in K_1$ by smoothness and compactness.) Then the argument leading to \eqref{Eqn:Local} applies to $p^{\nu,M'}_t(x,U_1^c,y)$ as well (and perhaps more easily, since $M'$ is compact). Because $p^{\nu}_t(x,y) =  p^{\nu,U_1}_t(x,y) +  p^{\nu}_t(x,U_1^c,y)$, $p^{\nu,M'}_t(x,y) =  p^{\nu,U_1}_t(x,y) +  p^{\nu,M'}_t(x,U_1^c,y)$, and $p^{\nu,U_1}_t(x,y)$ is the same whether it's viewed as a subset of $M$ or of $M'$, we conclude that, for any $\delta>0$ and any multi-index $\alpha$, there exists $t_0>0$, such that
\[
\lab\partial_y^{\alpha} p^{\nu}_t(x,y) - \partial_y^{\alpha} p^{\nu,M'}_t(x,y) \rab \leq \exp\lb -\frac{\lp d(x,y) +\frac{\eps'}{2}\rp^2 -\delta}{4t}\rb
\]
for all $x\in K_0$, $y\in B$, and $t\in (0,t_0]$. By Lemma \ref{Lem:Cov}, for each $N$, there is some $t_0>0$ such that
\begin{equation}\label{Eqn:Close}
\lab  \nabla^N_{V_1,\ldots,V_N} p^{\nu}_t(x,y) -  \nabla^N_{V_1,\ldots,V_N}  p^{\nu,M'}_t(x,y) \rab \leq \exp\lb -\frac{\lp d(x,y) +\frac{\eps'}{2}\rp^2 -\delta}{4t}\rb
\end{equation}
for any collection of unit vector fields $V_1,\ldots, V_N$ (acting on $y$) on $B$, and any $x\in K_0$, $y\in B$, and $t\in (0,t_0]$ (perhaps after decreasing $\delta$). 

Moreover, because $M'$ is compact and $p^{\nu,M'}_t$ is symmetric (so that we can exchange $x$ and $y$ derivatives) we know that $p^{\nu,M'}_t$ satisfies 
\begin{equation}\label{Eqn:ApproxSol}
\lab \nabla_y^N p_t^{\nu,M'}(x,y)\rab \leq D_N\lb \frac{d(x,y)}{t} +\frac{1}{\sqrt{t}}\rb^N p_t^{\nu,M'}(x,y)
\end{equation}
for $t\in(0,1]$, uniformly over $(x,y)\in M'$, for some sequence of constants.  Since both $p_t^{\nu,M'}$ and $p_t^{\nu}$ satisfy \eqref{Eqn:BN-A} and \eqref{Eqn:BN-Limit}, we see that they also satisfy \eqref{Eqn:Close} with $N=0$ (that is, such a bound holds for the heat kernels themselves, not just their derivatives). It follows that
\[
4t \log p^{\nu, M'}_t(x,y) \rightarrow -d^2(x,y) \quad\text{and}\quad\frac{p^{\nu,M'}_t(x,y)}{p^{\nu}_t(x,y)} \rightarrow 1\quad \text{as $t\searrow 0$,}
\]
uniformly for $x\in K_0$ and $y\in B$. In light of this and \eqref{Eqn:Close} (recall that $\delta$ can be chosen as small as needed), we can replace $p_t^{\nu,M'}$ with $p_t^{\nu}$
on both sides of \eqref{Eqn:ApproxSol} at the cost of possibly adjusting the $D_N$. We conclude that, for any $N$,  there is a constant $D_N$ and a $t_0>0$ such that
\begin{equation}\label{Eqn:LocalVersion}
\lab \nabla_y^N p_t^{\nu}(x,y)\rab \leq D_N\lb \frac{d(x,y)}{t} +\frac{1}{\sqrt{t}}\rb^N p_t^{\nu}(x,y)
\end{equation}
for all $t\in(0,t_0)$, $x\in K_0$, and $y\in B$.

Because $x$ is in the interior of $K_0$ and $B$ is open, and $(x,y)$ is an arbitrary point in $K$, we've proven that any point in $K$ has an open neighborhood on which \eqref{Eqn:LocalVersion} holds. By compactness of $K$, this means that we can find $D_N$ and $t_0$ such that \eqref{Eqn:LocalVersion} holds on all of $K$. Further, as already noted, we are free to let $t_0=1$ by smoothness, at the cost of again adjusting the $D_N$. Now by symmetry (of both $K$ and the heat kernel with respect to $\nu$) we can replace the $x$-derivatives by $y$-derivatives, at which point we use Lemma \ref{Lem:Constants} to replace $p_t^{\nu}$ with $p_t$ (since we no longer need the symmetry) and then again to deduce \eqref{Eqn:LogDer-Cor} from \eqref{Eqn:PDer}. This proves the theorem.
\end{proof}

\begin{proof}[Proof of Theorem \ref{THM:NonCut}]
As in the proof of Theorem \ref{THM:Main}, we assume $K$ is symmetric (recall that $\Cut(M)$ is symmetric, so this is compatible with the assumptions), and we construct $K_0\subset U_0\subset K_1\subset U_1\subset M$ in the same way.

The difference with the proof of Theorem \ref{THM:Main} is that now we use bump function to include $U_1$ in a manifold $M'$ for which the heat kernel satisfies the following expansion: there exists some $t_0>0$ and a function $c_0(x,y)$, smooth on a neighborhood of $K$ such that $c_0>0$ on $K$ and the heat kernel satisfies the expansion
\begin{equation}\label{Eqn:BA}
p^{\nu,M'}_t(x,y)= \frac{1}{t^{n/2}} e^{-\frac{d^2(x,y)}{4t}}\lp c_0(x,y)+r(t,x,y)\rp
\end{equation}
where the remainder $r(t,x,y)$ satisfies the bound
\[
\sup_{(t,x,y)\in (0,t_0]\times K} \lab \partial_x^{\alpha}\partial_y^{\beta} \partial^k_t r(t,x,y)\rab <\infty
\]
for all multi-indices $\alpha$ and $\beta$ and all $k=0,1,2,\ldots$. That this is possible follows from the well-known paper of Ben Arous \cite{BenArous}, although one should account for any topological difficulties in the inclusion, as done in \cite{WithLudovic}. Alternatively, Ludewig \cite{Ludewig1} proves that \eqref{Eqn:BA} holds for compact $M'$ (or even a general complete $M'$), when one considers the specialization from his vector bundle context to the scalar Laplacian. (Actually, both authors give an expansion to any number of terms, but we only need the leading coefficient $c_0$ at present.)

Either way, we have such an $M'$ so that \eqref{Eqn:Close} holds. Since the right-hand side of \eqref{Eqn:Close} is exponentially small compared to the remainder in \eqref{Eqn:BA}, we have that \eqref{Eqn:BA} holds for $p^{\nu}$ at the cost of altering the remainder but still having it satisfy the the same bound. Thus
\[
\log p^{\nu}_t(x,y)= -\frac{n}{2}\log t -\frac{d^2(x,y)}{4t}+ \log\lp c_0(x,y)+r(t,x,y)\rp .
\]
Taking derivatives yields
\[
-4t \nabla^N_y \log p^{\nu}_t(x,y) - \nabla^N_y d^2(x,y) = -4t \nabla^N_y \log\lp c_0(x,y)+r(t,x,y)\rp .
\]
Now since $c_0>0$ and $r$ has uniformly bounded derivatives, $\nabla^N_y \log\lp c_0(x,y)+r(t,x,y)\rp$ is uniformly bounded for $(t,x,y)\in (0,t_0]\times K$ and thus the right-hand side of the above is uniformly $O(t)$. By symmetry, we have that
\[
-4t \nabla^N_x \log p^{\nu}_t(x,y) -\nabla^N_x d^2(x,y) =O(t) \quad\text{as $t\searrow 0$}
\]
uniformly in $(x,y)\in K$. Because
\[
\log p_t(x,y) = \log p_t^{\nu}(x,y) + \log \frac{d \nu}{d \mu}(y) ,
\]
applying $-4t \nabla^N_x$ to both sides gives the theorem.

Finally, in regard to Remark \ref{Rmk:YDers}, we note that 
\[
4t \nabla^N_y \log p_t(x,y)  -4t\nabla^N_y \log p_t^{\nu}(x,y) = 4t \nabla^N_y \log \frac{d \nu}{d \mu}(y) .
\]
Since $\frac{d \nu}{d \mu}(y)$ is smooth and positive for $y\in \pi_2(K)$, the right-hand side is $O(t)$ uniformly, just as we saw for $\nabla^N_y \log\lp c_0(x,y)+r(t,x,y)\rp$ above. This shows that Theorem \ref{THM:NonCut} holds for either $x$ or $y$ derivatives, as desired.
\end{proof}

\begin{remark}\label{Rmk:Elton}
As mentioned, weaker versions of these theorems can be proven using results on Riemannian Brownian motion from the 90s. In particular, assume that, in Theorems \ref{THM:Main} and \ref{THM:NonCut}, $Z\equiv 0$ and $K$ satisfies the stronger assumption that there exists $\eps>0$ such that, for all $(x,y)\in K$, $d(x,\infty)+d(y,\infty)> d(x,y)+\eps$, where
\[\begin{split}
d(x,A) &= \inf\lc d(x,z): z\in A\rc \\
\text{and}\quad d(x,\infty) &= \sup\lc d(x, A) : \text{$A$ closed and $M \setminus A$ relatively compact}\rc .
\end{split}\]
In this case, we can use that \eqref{Eqn:LogDer} itself holds on compact manifolds, along with Lemmas \ref{Lem:Constants} and \ref{Lem:Cov} to move between partial derivatives and covariant derivatives as well as between logarithmic derivatives and ``regular'' derivatives of $p_t$, in place of L\'eandre's result \eqref{Eqn:LeandreCoarse2}. That gives the necessary input for Lemma \ref{Lem:NestedBalls}. Moreover, a pair of papers of Hsu \cite{HsuLocal, HsuIncomplete} can be used in place of the results of Bailleul-Norris \eqref{Eqn:BN-A} and \eqref{Eqn:BN-Limit}, giving the necessary input for Lemma \ref{Lem:Uc} with a larger $U_1$ than before, reflecting the weaker localization condition for $K$.

Given this, the proofs of Theorems \ref{THM:Main} and \ref{THM:NonCut} proceed in essentially the same way. In Theorem \ref{THM:Main} we localize \eqref{Eqn:LogDer} as proven in Hsu \cite{HsuLogDer} and Stroock-Turetsky \cite{StroockTuretsky} , with no need to refer to \cite{WithLudovic} for the case of non-zero $Z$, and similarly in Theorem \ref{THM:NonCut}, we can localize the results of \cite{Stroock-Malliavin} without using the expansion \eqref{Eqn:BA} to handle the case of non-zero $Z$, although at the cost of weakening the $O(t)$ rate of convergence.
\end{remark}


\section{Sharpness of the upper bound and explicit examples}\label{Sect:Misc}

In \cite{WithLudovic}, an expression for the leading (that is, the $1/t^N$ order) term in the asymptotics of an $N$th-order derivative of $\log p_t(x,y)$ is given. Briefly, if $\Gamma$ is the set of midpoints of minimal geodesics from $x$ to $y$, there is a probability measure $m_0$ on $\Gamma$ such that, for any smooth vector fields $Z_1,\ldots, Z^N$ in a neighborhood of $y$,
\begin{equation}\label{Eqn:Cumulants}
\lim_{t \searrow 0} 
 t^N Z_y^N \cdots Z^1_y \log p_{t}(x,y) = 
 \lp -\frac{d(x,y)}{2} \rp^N   \kappa^{m_0} \Big(Z^1_yd(\cdot,y),\ldots,  Z^N_yd(\cdot,y) \Big) .
\end{equation}
where $\kappa^{m_0}$ is the joint cumulant with respect to $m_0$ and the $Z^i_yd(\cdot,y)$ are viewed as random variables on $\Gamma$. Note that if $(x,y)\not\in \Cut(M)$, then $\Gamma$ is a single point, $m_0$ is a point mass, and any cumulant of order 2 or more is 0, which is consistent with Theorem \ref{THM:NonCut}. However, since this cumulant will generally be non-zero for $(x,y)\in \Cut(M)$, it's clear that the upper bound in \eqref{Eqn:LogDer} is sharp (that is, the constant is not explicitly given, but $1/t^N$ is the correct order of blow up) even on a compact Riemannian manifold (so this is not a matter of completeness or incompleteness). To underscore this, and also to illustrate the pointwise result just mentioned, we treat the simple example of the circle explicitly.

If we consider $\bS^1\cong \bR/\Zed$ with the standard heat kernel (that is, with $Z\equiv 0$), then for any $x$ and $y$ (understood mod 1), we have
\[
p_t^{\bS^1}(x,y) = \frac{1}{\sqrt{4\pi t}}\sum_{n=-\infty}^{\infty} e^{-\frac{(y-x+n)^2}{4t}} .
\]
Here the symmetry between $x$ and $y$ is explicit, and one can take spatial derivatives on either variable. By translation invariance, without loss of generality we can let $x=0$, in the sense that it is the image of 0 under the quotient map, in which case the cut locus of $x$ is the point $y=\frac{1}{2}$, again in the sense that it is the image of $\frac{1}{2}$ under the quotient map. More geometrically, $y$ is the antipodal point of $x$ and there are two minimizing geodesics from $x$ to $y$, corresponding to the two ways to go halfway around the circle.

To compute the log-Hessian of $p_t(0,1/2)$, let $y=\frac{1}{2}+\alpha$, for small $\alpha$ (which is trivially normal coordinates centered at $\frac{1}{2}$). Then explicit computation gives that,
\[\begin{split}
\partial_{\alpha} p_t^{\bS^1}\lp 0,\frac{1}{2}+\alpha \rp &=
\frac{1}{\sqrt{4\pi t}}\sum_{n=-\infty}^{\infty} -\frac{\frac{1}{2}+\alpha+n}{2t}e^{-\frac{\lp\frac{1}{2}+\alpha+n\rp^2}{4t}}  \quad\text{and}\\
\partial_{\alpha}^2 p_t^{\bS^1}\lp 0,\frac{1}{2}+\alpha \rp &=
\frac{1}{\sqrt{4\pi t}}\sum_{n=-\infty}^{\infty} \lb -\frac{1}{2t}+\frac{\lp\frac{1}{2}+\alpha+n\rp^2}{4t^2}\rb e^{-\frac{\lp\frac{1}{2}+\alpha+n\rp^2}{4t}} .
\end{split}\]
Letting $\alpha=0$, we see that only two terms from the sum contribute (up to an exponentially small error), corresponding to the two minimizing geodesics. In particular, for any $\eps>0$, we have
\[\begin{split}
p_t^{\bS^1}\lp 0,\frac{1}{2} \rp &= \frac{1}{\sqrt{4\pi t}}\lb 2e^{-\frac{1}{16t}} \rb +O\lp e^{-\frac{9-\eps}{16t}}\rp , \\
\partial_{\alpha} p_t^{\bS^1}\lp 0,\frac{1}{2} \rp &= 0  \quad\text{and} \\
\partial_{\alpha}^2 p_t^{\bS^1}\lp 0,\frac{1}{2} \rp &=
\frac{1}{\sqrt{4\pi t}} \lb -\frac{1}{2t}+\frac{1}{16t^2}\rb 2 e^{-\frac{1}{16t}} +O\lp e^{-\frac{9-\eps}{16t}} \rp.
\end{split}\]
Then since $\partial^2_{y}\log p_t(x,y) = \frac{\partial^2_{y} p_t(x,y)}{p_t(x,y)} -\lp\frac{\partial_{y} p_t(x,y)}{p_t(x,y)}\rp^2$, we have
\[
\partial^2_{x}\log p_t^{\bS^1}\lp 0,\frac{1}{2} \rp =\partial^2_{y}\log p_t^{\bS^1}\lp 0,\frac{1}{2} \rp = -\frac{1}{2t}+\frac{1}{16 t^2} +O\lp  e^{-\frac{8-\eps}{16t}}\rp
\]
for any $\eps>0$.

We note that here we have computed the logarithmic Hessian of the heat kernel directly (that is, by elementary methods, independent of any of the preceding results of this paper). In reference to \eqref{Eqn:Cumulants}, note that, by symmetry, $m_0$ assigns probability $1/2$ each to the 2 midpoints of the 2 minimizing geodesics from 0 to $1/2$, and thus the random variable $\partial_{\alpha} d\lp \cdot, \frac{1}{2}\rp$ under $m_0$ takes values $\pm1$ each with probability $1/2$. So the variance (note that the joint cumulant of 2 random variables is their covariance) in \eqref{Eqn:Cumulants} is 1, and we have
\[
\lim_{t\searrow0}t^2 \nabla_x^2 \log p^{\bS^1}_t\lp 0,\frac{1}{2}\rp = \frac{d^2\lp 0,\frac{1}{2}\rp}{4} = \frac{1}{16}
\]
using $\partial_{\alpha}\otimes\partial_{\alpha}$ as a basis for bilinear forms on $T_xM$, in agreement with the above.

Also, note that this is not special to this simple example. Any time (on any Riemannian manifold and for any $x$ and $y$ points such that the conditions of Theorem \ref{THM:Main} hold) there are exactly 2 non-conjugate minimizing geodesics between $x$ and $y$, one will have similar behavior for the logarithmic Hessian.

\section{The non-symmetric and sub-Riemannian cases}

Now suppose $Z$ is such that $\Delta$ no longer admits a symmetrizing measure. Then $y$-derivatives of the heat kernel can still be localized, and thus are seen to satisfy the same estimates as above, under either the condition on the distance to infinity in Theorem \ref{THM:Main} (if the asymmetry is ``not too bad'' in the sense of the sector condition of \cite{IsmaelNorris}) or the condition described in Remark \ref{Rmk:Elton} (in general). However, without the symmetry, $y$-derivatives cannot be transferred to $x$-derivatives. Indeed, when exchanging the role of the variables in general, one needs to take the adjoint of the operator (this is the relationship between the forward and backward Kolmogorov equations). The point is that the adjoint may have a potential term (that is, 0th-order term), which is not covered by the preceding arguments and is less easy to localize.

To quickly illustrate the point, we let $(u,v)$ be Cartesian coordinates on $\bR^2$. If we consider $\Delta=\partial_u^2+\partial_v^2+a(u,v)\partial_u+b(u,v)\partial_v$ then the (formal) adjoint, with respect to the standard Euclidean area, is $\Delta^* = \partial_u^2+\partial_v^2-(u,v)\partial_u-b(x,y)\partial_v-\lp \partial_u a\rp (u,v) - \lp \partial_v b\rp(u,v)$. Thus we now have a potential, and to find a stochastic representation for the heat kernel, corresponding to $\Delta^*$, we should look at the Feynman-Kac formula
\begin{equation}\label{Eqn:FK}
p_t(x,y) = \E^x \lb \exp\lp \int_0^t \lp \partial_u a+\partial_v b \rp\lp X_s\rp \, ds\rp \rb
\end{equation}
where $X_t$ is the diffusion with generator $\partial_u^2+\partial_v^2-(u,v)\partial_u-b(x,y)\partial_v$ started from $x$. On a non-compact manifold, the integrand on the right-hand side need not be bounded, so even if the probability of a path leaving some open $U$ and eventually arriving at $y$ can be estimated to be small, as above, the contribution of such a path to $p_t$ is weighted by the Feynman-Kac factor, which depends on the entire trajectory and which we have no a priori control over. Similarly, trying to treat the potential via an h-transform runs into the problem that it's not clear in general why the heat equation for the adjoint should admit a positive solution. Indeed, it is not clear that we should expect \eqref{Eqn:LogDer} to hold on a non-compact Riemannian manifold without some additional assumptions to control the potential. In that spirit, one can see  \cite{XM-ParabolicSch} for recent results in this direction on complete (but not necessarily compact) Riemannian manifolds under additional geometric assumptions (for example, curvature bounds, the existence of a pole, etc.). (Note also that the case where the path integral in \eqref{Eqn:FK} only depends on the endpoints and thus can be localized is exactly the case when $Z$ admits a potential, and this is directly related to the density of the symmetrizing measure.)

In a different direction, suppose $p_t(x,y)$ is the heat kernel associated to a sub-Laplacian on a sub-Riemannian manifold $M$. Then the localization for $p_t$ and its $y$-derivatives still works (and thus under a symmetry assumption, it also works for $x$-derivatives), but we don't have a global estimate to localize. That is, no estimates comparable to \eqref{Eqn:LogDer} are known to hold for compact sub-Riemannian manifolds (or for structures on $\bR^n$ given by bounded vector fields with bounded derivatives of all orders, as in the work of L\'eandre and Ben Arous used above). More precisely, such bounds are known to hold away from abnormal minimizers (see \cite{WithLudovic}), but on a properly sub-Riemannian manifold (that is, not one that happens to be Riemannian), the diagonal is always abnormal, and thus one does not have appropriate estimates on any $K$ that intersects the diagonal. (Note that more than the asymptotic results of \cite{Trelat} are needed in the present context.) Derivative bounds for the special case of H-type groups were established in \cite{Nate} and \cite{HQLi}, including bounds for the log-gradient of the heat kernel in \cite{HQLi}, but this comes nowhere near the complete generality we have in the Riemannian situation.

That said, note that the semi-martingale property for the bridge process on compact, step-2 sub-Riemannian manifolds was established in \cite{XM-HypoBridge}, so there is hope for future advances.


\section{Acknowledgements}

This work was partially supported by a grant from the Simons Foundation (\#524713 to Robert Neel).

\providecommand{\bysame}{\leavevmode\hbox to3em{\hrulefill}\thinspace}
\providecommand{\MR}{\relax\ifhmode\unskip\space\fi MR }
\providecommand{\MRhref}[2]{%
  \href{http://www.ams.org/mathscinet-getitem?mr=#1}{#2}
}
\providecommand{\href}[2]{#2}

\end{document}